\documentclass[10pt]{article}
\usepackage{amsthm, amsmath,amssymb,tabularx,amsthm}
\usepackage{longtable}
\usepackage[T1]{fontenc}
\usepackage[latin1]{inputenc}
\newtheorem{theorem}{Theorem}[section]
\usepackage{url} 
\usepackage{icomma} 
\usepackage{calc}
\usepackage[hang]{caption}
\usepackage{varioref}
\usepackage{booktabs,longtable}
\numberwithin{equation}{section}
\usepackage{algpseudocode}
\usepackage{MnSymbol}
\algrenewcommand\algorithmicindent{0.7em}
\begin{document}

\title{A new class of maximal partial spreads in $\mathrm{PG}(4,q)$}

\author{Sandro Rajola and Maurizio Iurlo}

\date{\,}

\maketitle

\begin{abstract}
We construct a class of maximal partial line spreads in
$\mathrm{PG}(4, q)$, that we call \textit{q-added} maximal partial line spreads. We obtain them
by depriving a line spread of a hyperplane of some lines and adding $q+1$ pairwise skew
lines not of the hyperplane for each removed line. We do this through a
theoretical way for every value of $q$, and by a computer search for
 $q\leq16$. More precisely we prove that for every
$q$ there are $q$-added maximal partial line spreads from the size $q^{2}+q+1$ to the size $q^{2}+(q-1)q+1$,
while by a computer search we get larger cardinalities.
\end{abstract}

\noindent \textbf{Keywords:}\,\,\,Maximal partial line spreads - Computer search

\vspace{0,2cm}\noindent \textbf{AMS Classification:}\,\,\,51E14

\section{Introduction}
A \emph{partial line spread} $\mathcal{F}$ in $\mathrm{PG}(4, q)$,
the projective space of dimensions four  over the Galois
field $\mathrm{GF}(q)$ of order $q$,
 is a set of pairwise
skew lines. We say that $\mathcal{F}$ is \emph{maximal} if it cannot be extended to a larger
partial line spread. A \emph{line spread} in $\mathrm{PG}(3, q)$, generally in $\mathrm{PG}(n, q)$, $n$ odd, is a set of pairwise skew lines covering the space.

Maximal partial line spreads (from now on MPS) in $\mathrm{PG}(4, q)$ have
been investigated by several authors, but little is known about them  (see \cite{gaszon2}, \cite{gova}).
The smallest examples are the spreads in hyperplanes (of size
$q^{2}+1$). 

A. Beutelspacher determined the largest
examples, which are of size $q^{3}+1$,  and also found examples inside the interval $\left[
q^{2}+1, q^{2}+q\sqrt{q}-\sqrt{q}\right]$ (see \cite{beut1}). Afterwards, we only have a
density result by J. Eisfeld, L. Storme and P. Sziklai \cite{eisf}, precisely the interval
[$q^{3}-q+3, q^{3}+1$].

Here we construct a particular class of MPS in $\mathrm{PG}(4, q)$. We do
it for every value of $q$ by using theoretical methods and for  $q\leq 16$ by a computer search. 

We start from a line spread $\mathcal{F_{\mathcal{H}}}$ of a hyperplane
$\mathcal{H}$ of $\mathrm{PG}(4, q)$.  
We begin by depriving $\mathcal{F_H}$ of a line $r_1%
\in\mathcal{H}$. Starting from $\mathcal{F}_\mathcal{H}\backslash\{r_1\}$
we construct a new maximal partial line spread $\mathcal{F}_\mathcal{H}^{\prime}$, by
adding $q+1$ pairwise skew lines not of $\mathcal{H}$ and covering the line $r_1$.
Afterwards we deprive $\mathcal{F}_\mathcal{H}^{\prime}$ of a line $r_2\in\mathcal{H}$ and add $q+1$ pairwise skew lines not of
$\mathcal{H}$, meeting $r_2$ and not meeting the previous added lines, obtaining a MPS
$\mathcal{F}_\mathcal{H}^{\prime\prime}$, and so on.

We call $(\mathcal{F_{\mathcal{H}}}, q)$-\textit{added maximal partial line spread}, or briefly $q$-\textit{added maximal partial line spread}, the MPS
obtained in this way.
We prove that this construction can be repeated for $q-1$ times, for every value of $q$.
So we get $q$-added MPS of size $q^{2}+kq+1$, for every integer $k=1, \ldots, q-1$.

By a computer search the previous construction can be repeated for a larger
number of times, for  $q\leq 16$. More precisely, we construct all the 
$q$-added MPS with size between $q^{2}+q+1$ and $q^{2}+k(q)q+1$,
where $k(3)=4$, $k(4)=7$, $k(5)=10$, $k(7)=19$, $k(8)=26$, $k(9)=33$, $k(11)=46$, $k(13)=62$ and $k(16)=87$. 

We remark that the times for the computer search are restricted: for instance,
the case $q=11$ needs a time less than 80 seconds, by a notebook with processor Intel Core i5-430M, 4 GB RAM.

\section{A geometric construction of the $q$-added maximal partial line spreads}

We start by proving the following theorem.

\begin{theorem}\label{teo2.1}

In $\mathrm{PG}(4, q)$, $q$ a prime power, let $S$ be a hyperplane
and $X$ a point of $S$. Let $\mathcal{L}$ be a set of pairwise
skew lines not of $S$ and not through $X$, such that $\left|\mathcal{L}\right|<q^{2}$.
Then there is a line through $X$ not of $S$ skew with every line of $\mathcal{L}$.

\end{theorem}

\begin{proof}

Through the point $X$ there are $\theta_{3}-\theta_{2}=q^{3}$ (where $\theta_r=\sum_{i=0}^{r}q^i$) lines
not of $S$ and therefore having only the point $X$ in common with $S$. Let $L$ be the following point
set:
\[
L=\bigcup_{\ell\in\mathcal{L}}\ell-S.
\]
Evidently, we have:
\begin{equation}
\left|L\right|=q\left|\mathcal{L}\right|.	\label{th1}
\end{equation}
Assume that every line through $X$ and not of $S$ meets some lines
of $\mathcal{L}$. Since there are $q^3$ lines 
having only the point $X$ in common with $S$ and since such lines can meet $\bigcup_{\ell\in\mathcal{L}}\ell$ only 
at points of $L$, we have 
\begin{equation}
\left|L\right|\geq q^{3}.	\label{th2}
\end{equation}
By (\ref{th1}) and (\ref{th2}) we get
\begin{equation}
\left|\mathcal{L}\right|\geq q^{2}.\label{th3}
\end{equation}
The inequality (\ref{th3}) is a contradiction, since $\left|\mathcal{L}\right|<q^{2}$.
The contradiction proves that there is a line through $X$, not of
$S$ and not meeting any line of $\mathcal{L}$. So the theorem is proved.

\end{proof}

Now let $\mathcal{F}$ be a spread of a hyperplane $S$ 
and let $r_{1}, r_{2}, \ldots, r_{q-1}$ be $q-1$ 
 lines of $\mathcal{F}$. By using Theorem \ref{teo2.1} we find $q+1$
mutually skew lines, $r_{1}^{1}, r_{1}^{2}, \ldots, r_{1}^{q+1}$, not
of $S$ and covering $r_{1}$. The line set
\[
\mathcal{F}_{1}=\left(\mathcal{F}-\left\{ r_{1}\right\} \right)\bigcup\left\{ r_{1}^{1}, r_{1}^{2}, \ldots, r_{1}^{q+1}\right\} 
\]
is a $q$-added MPS with size $q^{2}+q+1$. By using Theorem \ref{teo2.1} we find
$q+1$ mutually skew lines, $r_{2}^{1}, r_{2}^{2}, \ldots, r_{2}^{q+1}$
not of $S$, covering $r_2$ and not meeting $r_{1}^{1}, r_{1}^{2}, \ldots, r_{1}^{q+1}$.
The line set 
\[
\mathcal{F}_{2}=\left(\mathcal{F}_{1}-\left\{ r_{2}\right\} \right)\bigcup\left\{ r_{2}^{1}, r_{2}^{2}, \ldots, r_{2}^{q+1}\right\} 
\]
is a $q$-added MPS, with size $q^{2}+2q+1$. Theorem \ref{teo2.1} allows us
to construct $q$-added MPS up to the cardinality $q^{2}+\left(q-1\right)q+1$,
by covering all the lines $r_1, r_2, \ldots, r_{q-1}$.

So we get the following theorem.

\begin{theorem}\label{teo2.2}

In $\textnormal{PG}(4, q)$, $q$ a prime power, there are $q$-added
maximal partial line spreads of size $q^{2}+kq+1$, for every integer $k=1, 2, \ldots, q-1$.

\end{theorem}

\section{Computer search of $q$-added maximal partial line spreads in $\mathrm{PG}(4, q)$, $q$ a prime}

\subsection{The algorithm}
We are able to construct the Pl\"{u}cker coordinates of the lines of a spread
$\mathcal{F}$ of $\mathrm{PG}(3, q)$, by using the programs for the computer search of MPS
in $\mathrm{PG}(3, q)$ contained in \cite{iura,iura2}.

In the case $q=p^h$, we use the following spread (see \cite{Hir}, 17.3.3).

Let $q=p^{h}$, with $h>1$ and let $x^{p+1}+bx-c$ be a polynomial
without roots in $\mathrm{F=GF}(q).$ Then the set 
\begin{equation*}
\left\{ \left\langle \left(1, 0, 0, 0\right),\,\left(0, 1, 0, 0\right)\right\rangle \right\} \cup\left\{ \left\langle \left(z, y, 1, 0\right),\,\left(cy^{p}, z^{p}+by^{p}, 0, 1\right)\right\rangle \parallel\left(y, z\right)\in\mathrm{F}^{2}\right\} 
\end{equation*}
is an aregular spread of $\mathrm{PG}(3, q)$.
 
In the case $q$ a prime, with $\gcd(q+1, 3)=3$, we use either the spread 
obtained by A.\thinspace A.\thinspace Bruen and J.\thinspace W.\thinspace P.\thinspace Hirschfeld and formed by
tangent lines, imaginary chords and imaginary axes of a twisted cubic
(see \cite{BH}), or a spread obtained by a computer search. In the case 
$q$ a prime, but not with $\gcd(q+1, 3)=3$, we use a spread obtained by a computer search.

Starting from the Pl\"{u}cker
coordinates $\left(p_{01}, p_{02}, p_{03}, p_{12}, p_{13}, p_{23}\right)  $ of
the lines of $\mathcal{F}$ in $\mathrm{PG}(3, q)$, we consider the 10-tuples%
\[
\left( p_{01}, p_{02}, p_{03}, p_{04}=0, p_{12}, p_{13}, p_{14}=0, p_{23}%
, p_{24}=0, p_{34}=0\right)  .
\]
It is easy to check that the above 10-tuples are the Pl\"{u}cker coordinates
of pairwise skew lines of $\mathrm{PG}(4, q)$ and that the hyperplane
$\mathcal{H}$ of equation $x_{4}=0$ contains all of them.
So we obtain a spread $\mathcal{F_H}$ of $\mathcal{H}$, which is a
maximal partial line spread of $\mathrm{PG}(4, q).$

We begin by depriving $\mathcal{F_H}$ of a line $r_1%
\in\mathcal{H}$. Starting from $\mathcal{F}_\mathcal{H}\backslash\{r_1\}$
we construct a new maximal partial line spread $\mathcal{F}_\mathcal{H}^{\prime}$, by
adding $q+1$ pairwise skew lines not of $\mathcal{H}$ and covering the line $r_1$.
Afterwards we deprive $\mathcal{F}_\mathcal{H}^{\prime}$ of a line $r_2\in\mathcal{H}$ and add $q+1$ pairwise skew lines not of
$\mathcal{H}$, meeting $r_2$ and not meeting the previous added lines, obtaining a MPS
$\mathcal{F}_\mathcal{H}^{\prime\prime}$, and so on.

\subsection{Result check}

We remark that the algorithm for the construction of the Pl\"{u}cker
coordinates is similar to what we used in two previous articles (see
\cite{iura,iura2}), where we constructed the Pl\"{u}cker
coordinates of the lines of $\mathrm{PG}(3, q)$.

Obviously, we verify the
correctness of our construction by several tests.

In particular, to verify the construction of the Pl\"{u}cker coordinates 
and the writing of the incidence line conditions, we write a (very simple)
program that calculates,
for every line $\ell$, the number $n(\ell)$ of the lines
meeting $\ell$. We do it entirely for $q=2, 3, 4, 5, 7, 8$ and partially for
$q=9, 11, 13, 16$ and we always find  $n(\ell)=\left(  q^{3}+q^{2}+q\right)
\left(  q+1\right)  +1$. The
number $n(\ell)$ has been calculated 	
approximately one million of times. This kind of test guarantees the correctness
of the construction of the Pl\"{u}cker coordinates and the writing of the
incidence relations.

Concerning the correctness of the obtained maximal partial spreads, we do the
following tests.

Firstly, we check that the $q^{2}+1$ lines of $\mathcal{F}$, that we
use as initial partial spread, form a set of pairwise skew lines and so a
spread of $\mathcal{H}$, which trivially is a maximal partial spread of
$\mathrm{PG}(4, q).$ 

Secondly, we test our biggest $q$-\thinspace added MPS by checking that
its lines are pairwise skew and that all the added lines are not of
$\mathcal{H}$. The first property is verified either through a macro in Microsoft Excel 
(in the case $q$ a prime) or with a  program in C language ($q$ a prime or not). 
%_______________________________________________________________________________________________________________________ 
The macro of Microsoft Excel is similar to the
macro used in \cite{iura,iura2}.  The
other property is verified by using very easy instructions,
which check that at least one of the Pl\"{u}cker coordinates $p_{04},$
$p_{14},$ $p_{24}$ and $p_{34}$ is different from zero.

Obviously, we test the ``test program''. We include some set $\mathcal{L}$ of lines of $\mathrm{PG}(4, q)$ 
and the program calculates the number $n(\ell)$, for each line of $\mathcal{L}$, and the number $n\left(\mathcal{L}\right)$ 
of lines of $\mathrm{PG}(4, q)$ meeting a line of $\mathcal{L}$.
\begin{enumerate}
\item We include in the test program $n$ coincident lines  and we always obtain the numbers
$n\left(\mathcal{L}\right)=\left(q^{3}+q^{2}+q\right)\left(q+1\right)+1$ and
$n(\ell)=n$, for every line.
\item We include in the test program two skew lines and the previous numbers
in this case are $n\left(\mathcal{L}\right)=2\left(\left(q^{3}+q^{2}+q\right)\left(q+1\right)+1\right)-\left(q+1\right)^{2}$ and $n(\ell)=1$, for each the two lines.
\item We include a spread $\mathcal{L}$ of a hyperplane of $\mathrm{PG}(4, q)$
(spread already used and tested in $\mathrm{PG}(3, q)$, see \cite{iura,iura2})
obtaining in this case $n(\ell)=1$, for each line, and $n\left(\mathcal{L}\right)=\theta_4\theta_3/\theta_1$,
that is the number of lines
of $\mathrm{PG}(4, q)$.
%_________________________________________________________________________________________________
After
this, we tested the spread in the following way: the spread is deprived of a line and the program
replies that it is not maximal (the previous numbers are $n(\ell)=1$ and $n\left(\mathcal{L}\right)<\theta_4\theta_3/\theta_1$). Furthermore, we tested the set
of lines
obtained by adding a line to the spread and the program
answers that it is not a set of pairwise skew lines and $n\left(\mathcal{L}\right)=\theta_4\theta_3/\theta_1$.
%_________________________________________________________________________________________________
\end{enumerate}

Concerning the programs which construct MPS in $\mathrm{PG}(4, p^{h})$,
with $p$ a prime and $h>1$, we remark that the used tables of sum and product are the
same
used and tested in \cite{iura2}. 

Furthermore, for $\mathrm{PG}(4, 2)$ and $\mathrm{PG}(4, 3)$ we construct
MPS either by using operations $\operatorname{mod}p$ or the tables of sum
and product, and we obtain the same results.

In addition to this we remark that the program never gives results against the
theory. In particular, in the case $\mathrm{PG}(4, 2)$, for which there is a
complete characterization of the MPS, the program
constructs MPS of size 7 and 9, according to the above characterization 
which asserts that in $\mathrm{PG}(4, 2)$ the only cardinalities for the MPS are 5, 7 and 9.

\section{Results}
In this paper we obtain MPS $\mathcal{F}$ of cardinality
$q^{2}+kq+1,$ with $k$ integer, which assumes all the values from 1 to the
maximum value $k_{\max}.$ In the following table we report the values of $q,$
$k_{\max}$ and the minimum and the maximum values of $\left\vert
\mathcal{F}\right\vert .$

\begin{longtable}{cccc}
\caption{}
\label{tab:longtable2}\\
\toprule
$q$ & $k_{\max}$ & $\left\vert \mathcal{F}_{\min}\right\vert =q^{2}+q+1$ & $\left\vert \mathcal{F}_{\max}\right\vert =q^{2}+k_{\max}q+1$ \\ 
\hline 
\endhead
\endfoot
3 & 4 & 13 & 22  \tabularnewline \hline 
4 & 7 & 21 & 45 \tabularnewline \hline
5 & 10 & 31  & 76 \tabularnewline \hline
7 & 19 & 57 & 183 \tabularnewline \hline
8 & 26 & 73 & 273  \tabularnewline \hline
9 & 33 & 91 & 379 \tabularnewline \hline
11 & 46  & 133 & 628 \tabularnewline \hline
13 & 62 & 183 & 976  \tabularnewline \hline
16 & 87 & 273 & 1649  \\
\bottomrule
\end{longtable}

We remark that some above values are already obtained by theoretic ways.

\section{Conclusion}
This work began by a computer search which allowed us to find several results and stimulated the theoretical research the results of which
are given by Theorem \ref{teo2.2}. 

The results of computer search prove that there is still more to be discovered. The analysis of our results
makes you think that the number of the $q$-added maximal partial line spreads increases with $q$.

We have realized that it is possible to study the same problem for larger values of the order $q$.

\begin{center}
--------------------
\end{center}
\noindent Maurizio Iurlo\\
\noindent Largo dell'Olgiata, 15/106/1C\\
\noindent 00123 Roma \\
\noindent Italy\\
\noindent maurizio.iurlo@istruzione.it\\
\noindent http://www.maurizioiurlo.com \\
\quad\\
\noindent Sandro Rajola\\
\noindent Istituto Tecnico per il Turismo ``C. Colombo''\\
\noindent Via Panisperna, 255\\
\noindent 00184 Roma\\
\noindent Italy\\
\noindent sandro.rajola@istruzione.it

\end{document}